\documentclass[final]{dmtcs-episciences}

\usepackage{mathrsfs}
\PassOptionsToPackage{hyphens}{url}\usepackage{hyperref}
\usepackage{tikz}
\usetikzlibrary{arrows,shapes,positioning,calc}
\usepackage{amsmath,amsthm,amssymb}
\renewcommand*{\bar}[1]{\overline{#1}}
\usepackage{graphicx}
\usepackage[colorinlistoftodos]{todonotes}
\usepackage{underscore}
\DeclareMathOperator{\optup}{ub}
\DeclareMathOperator{\optlow}{lb}
\usepackage[ruled,vlined,linesnumbered]{algorithm2e}
\newcommand{\R}{\mathbb{R}}

\newcommand{\Z}{\mathbb{Z}}
\newtheorem{theorem}{Theorem}
\newtheorem{conjecture}{Conjecture}

\newtheorem{corollary}{Corollary}[theorem]
\newtheorem{lemma}{Lemma}
\theoremstyle{remark}

\theoremstyle{definition}

\usepackage[utf8]{inputenc}
\usepackage{subfigure}

\usepackage[round]{natbib}

\author{Thomas Bellitto\affiliationmark{1}\thanks{This author is supported by the European Research Council (ERC) under the European Union’s Horizon 2020 research and innovation programme Grant Agreement 71470.}
  \and Arnaud P\^echer\affiliationmark{2}
  \and Antoine S\'edillot\affiliationmark{3}}
\title{On the density of sets of the Euclidean plane avoiding distance 1}
\affiliation{
  Faculty of Mathematics, Informatics and Mechanics, University of Warsaw, Poland\\
  Université de Bordeaux, INRIA, LaBRI, France\\
  Université de Paris, France}
\received{2019-01-31}

\revised{2020-11-16}

\accepted{2021-03-22}

\publicationdetails{23}{2021}{1}{8}{5153}
\begin{document}
\maketitle
\begin{abstract}
  A subset $A \subset \R^2$ is said to avoid distance $1$ if: $\forall x,y \in A, \left\| x-y \right\|_2 \neq 1.$ In this paper we study the number $m_1(\R^2)$ which is the supremum of the upper densities of measurable sets avoiding distance 1 in the Euclidean plane. Intuitively, $m_1(\R^2)$ represents the highest proportion of the plane that can be filled by a set avoiding distance 1. This parameter is related to the fractional chromatic number $\chi_f(\R^2)$ of the plane.
	
	We establish that $m_1(\R^2) \leq 0.25647$ and $\chi_f(\R^2) \geq 3.8991$.

\end{abstract}

\section{Introduction}

In the normed vector space $(\R^n,\left\| \cdot \right\|)$, a subset $A \in \R^n$ is said to \emph{avoid distance 1} if for all $x,y \in A$, we have $\left\| x-y\right\|$. In this work, we study $m_1(\R^n,\left\|\cdot\right\|)$ defined as the supremum of the upper densities of Lebesgue-measurable sets avoiding distance 1 (intuitively, the proportion of the space covered by such sets, a proper definition is given in Subsection \ref{sub:sets}). This number was introduced in \cite{LarmanRogers}. The aim of this work is to study the case of the Euclidean plane $(\R^2,\left\|\cdot\right\|_2)$: we denote $m_1(\R^2,\left\|\cdot\right\|_2)$ by $m_1(\R^2)$. 

A natural way to build a set avoiding distance 1 suitable for any normed space starts from a tiling of unit balls.
Let $\Lambda$ be a subset of $\R^n$ such that if $x \neq y$ are in $\Lambda$, the open unit balls $B(x,1)$ and $B(y,1)$ do not overlap. Then \[A = \bigcup_{\lambda \in \Lambda} B\left(\lambda,\frac{1}{2}\right),\] is a set avoiding distance $1$. The optimal density of a circle tiling in the Euclidean plane is about $0.9069$ and thus leads to the lower bound $m_1(\R^2)\geq 0.2267$ (see Figure \ref{fig:packing}):

\begin{figure}[!h]
\begin{minipage}{0.5\textwidth}
		\centering\includegraphics[width=0.5\linewidth]{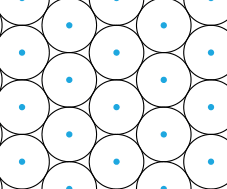}
		\caption{The points of $\Lambda$.}
\end{minipage}
\hfill
\begin{minipage}{0.5\textwidth}
		\centering\includegraphics[width=0.5\linewidth]{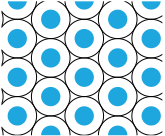}
		\caption{A set $A$ that avoids distance 1.}\label{fig:packing}
\end{minipage}
\end{figure}

By improving this method, Croft managed to build a set of density $0.2293$ in \cite{Croft} which is the best lower bound known up to now.

Prior to our work, the best known upper bound was equal to $0.258795 \ldots$ and is due to  Keleti, Matolcsi, de Oliveira Filho and Ruzsa (see \cite{Keleti}). A long-term objective would be to prove a conjecture by Erdős presented in \cite{Szekely} that states that:

\begin{conjecture}
\label{conj:Erdos}
$$m_1(\R^2) < \frac{1}{4}.$$
\end{conjecture}

This problem has also been studied in higher dimensions. Moser, Larman and Rogers generalized Erdős' conjecture as follows: 
$$\forall n\geq 2 ,m_1(\R^n) < \frac{1}{2^n}.$$

A weaker result has been proved in \cite{Keleti}. A subset $A \subset \R^n$ has a block structure if it is the disjoint union $A = \bigcup_{i \in I} A_i$ such that $\left\| x-y\right\| < 1$ if $x$ and $y$ belongs to the same block and $\left\| x-y\right\| > 1$ otherwise. The authors have proved that a set avoiding distance 1 that has block structure has density strictly smaller than $\frac{1}{2^n}$.

For higher dimensions, as shown in \cite{Christine2015}, the best asymptotic bound is $\left(1 + o(1) \right)(1.2)^{-n}$, thus without any additional assumption, the current bounds are pretty far from the conjectured $2^{-n}.$

The special case of norms such that the unit ball tiles $\R^n$ by translation has also been studied. In this case, the previous method provides a set avoiding distance 1 of density $\frac{1}{2^n}$. 

Bachoc and Robins conjectured that this construction is optimal:

\begin{conjecture}[Bachoc, Robins]
Let $\left\| \cdot \right\|$ be a norm on $\R^n$ such that the unit ball tiles $\R^n$ by translation, then
$$m_1(\R^n, \left\| \cdot \right\|) = \frac{1}{2^n}.$$ 
\end{conjecture}

This conjecture is proved in dimension 2:

\begin{theorem}[\cite{AD1}, \cite{philippe}] 
Let $\left\| \cdot \right\|$ be a norm on $\R^2$ such that the unit ball tiles $\R^2$ by translation, then
$$m_1(\R^2, \left\| \cdot \right\|) = \frac{1}{4}.$$
\end{theorem}

A possible approach to the study of $m_1(\R^2)$ uses graph theory, more precisely unit-distance graphs. Given a normed space $(\R^n,\left\| \cdot \right\|)$, a \emph{unit-distance graph} is a graph whose vertices are points of $\R^n$ and where two vertices are adjacent if and only if they are at distance exactly one.



A set coloring of a graph $G = (V,E)$ is a function $c$ such that each vertex is associated to a set and if $(u,v) \in E, c(u) \cap c(v) = \emptyset$. An $(a:b)$-coloring of $G$ is a set coloring such that $\left|c(v)\right| \geq b $ for all vertex $v$ and $\left| \cup_{v \in V} c(v) \right| \leq a$.
Given an integer $b$, the $b$-fold chromatic number of $G$, $\chi_b(G)$, is the least integer $a$ such that there exists an $(a:b)$-coloring of $G$. The \textit{fractional chromatic number} of a graph $G$ (finite or not), $\chi_f(G)$, is then defined as 
$$\chi_f(G) = \inf_{b} \frac{\chi_b(G)}{b}.$$ 
If we consider the particular case of the unit-distance graph of $\R^2$ whose vertices are all the points of $\R^2$, its fractional chormatic number, denoted by $\chi_f(\R^2)$ is called the \emph{fractional chromatic number of the plane}.

Prior to our work, the best known lower bound was

$$\frac{382}{102} = 3.745... \leq \chi_f(\R^2).$$

This lower bound is claimed by Exoo and Ismailescu in the yet unpublished paper \cite{ExooIsmailescu} by constructing a $72$-vertices unit-distance graph of fractional chromatic number $\frac{382}{102}$. The best published bound was $\frac {76}{21}=3.619...$ and was obtained in \cite{chifR2}.

If $G$ is a unit-distance graph of $(\R^n,\left\| \cdot \right\|)$, the following fundamental results holds (Lemma \ref{lemchif} and Lemma \ref{lem:alphastar}):
\begin{align}
  m_1(\R^n,\left\| \cdot \right\|) \leq \frac{1}{\chi_f(G)}.  
\end{align}

Furthermore, we obviously have
\begin{align}
  \chi_f(G) \leq  \chi_f(\R^n).
\end{align}





In this work, we build a finite unit-distance graph with fractional chromatic number less than $1999983/512933\geq  3.8991$ (see Section \ref{sec:bound}), hence our main result is: 

\begin{theorem}
	\label{th:main}
\begin{eqnarray}
 m_1(\R^2) & \leq & 0.25647. \\
 3.8991. & \leq & \chi_f(\R^2). 
\end{eqnarray}

\end{theorem}

which improves the previous best bounds from \cite{Keleti} and \cite{chifR2}.

The standard algorithm to compute the fractional chromatic number of a finite graph is a pure linear program with as many constraints as maximal stable sets of the graph. It is not efficient in practice for the unit-distance graphs that we study in this work, as they have exponentially many maximal stable sets.

To overcome this difficulty, we introduce an algorithm combining iteratively a mixed integer linear program and a pure linear program, which turned out to be practical for unit-distance graphs with up to 600 vertices.

This paper is organized as follows: 
\begin{itemize}
	\item Section \ref{sec:prelim} contains preliminaries;
	\item in Section \ref{sec:algo}, we give our algorithm to compute the optimal weighted independence ratio of any finite graph, establish its correctness and exhibit some optimizations for unit-distance graphs;
	\item in Section \ref{sec:proof}, we exhibit a graph yielding the bound of Theorem \ref{th:main}, and explain how we constructed it.
\end{itemize}

\section{Preliminaries}
\label{sec:prelim}

In this section, we recall some basic graph terminology and introduce a key graph parameter for this work: the optimal weighted independence ratio (subsection \ref{sub:chif}).
We formally define the parameter $m_1(\R^n,||.||)$ for any norm and establish that it is upper bounded by the optimal weighted independence ratio of any unit-distance graph (subsection \ref{sub:sets}).

\subsection{Fractional chromatic number and optimal weighted independence ratio}
\label{sub:chif}

Let $G = (V,E)$ be a graph. An independent set (resp. clique) of $G$ is a set of pairwise non-adjacent (resp. adjacent) vertices of $G$. A maximal independent set (resp. maximal clique) is an independent set which is not properly contained in an other one. We denote by $\mathcal{I}(G)$ be the set of all independent sets in $G$. It is well known that the fractional chromatic number of $G$, $\chi_f(G)$, is the solution of the linear program
\begin{equation}
\label{eq:fractlinprog}
\begin{array}{ll@{}ll}
\text{minimize}  & \displaystyle\sum\limits_{I \in \mathcal{I}(G)} x_I, &\\
\text{subject to}& \displaystyle\sum\limits_{I \in \mathcal{I}(G), y \in I} & x_I \geq 1, & \forall I \in \mathcal{I}(G),\\
\text{and}       &                                                &x_I \geq 0, &\forall I \in \mathcal{I}(G).
\end{array}
\end{equation}


A \emph{weighted graph} is a triplet $(V,E,w)$ where $(V,E)$ is a graph and $w : V \rightarrow \R_{+}$ is a not identically 0 function on the vertices called a \emph{weight distribution}. We denote by $W$ the \emph{set of weight distributions} on $G = (V,E)$, and for $v \in V$, $w(v)$ is called the \emph{weight} of the vertex $v$. The weight of a vertex set $A \subset V$ is naturally defined as $w(A) = \sum_{v \in A}w(v)$ and $w(V)$ is called the \emph{total weight} of the graph.

The \emph{weighted independence number} of a finite weighted graph $G_w$ is denoted by $\alpha(G_w)$ and is the maximum weight of an independent set of $G$. The \emph{weighted independence ratio} of the weighted graph is $\bar{\alpha}(G_w) = \frac{\alpha(G_w)}{w(V)}$. 

We define the \emph{optimal weighted independence ratio} of a finite graph $G$, denoted by $\alpha^{*}(G)$, as the infimum of all weighted independence ratios:

\begin{align}
\label{eq:alphastar}
\alpha^{*}(G) = \displaystyle\inf_{w \in W}\bar{\alpha}(G_w)
\end{align}

Then we have this basic equality:

\begin{lemma}\label{lemchif}
For every graph $G$,  $\alpha^{*}(G) = \frac{1}{\chi_f(G)}$.
\end{lemma}

\begin{proof} 

As this equality is crucial for our work, we give its short proof for the sake of completeness.
	
We have to prove that 
$$\chi_f(G) = \displaystyle\sup_{w \in W} \frac{1}{\bar{\alpha}(G_w)}.$$

By the strong duality theorem, $\chi_f(G)$ is equal to the fractional clique number  $\omega_f(G)$, which is the solution of the dual of the linear program (\ref{eq:fractlinprog}) \cite{DBLP:journals/tit/Shannon56,DBLP:journals/combinatorica/GrotschelLS81}:

\begin{equation}
\label{eq:fractdualLP}
\begin{array}{ll@{}ll}
\text{maximize}  & \displaystyle\sum\limits_{v \in V} y_v, &\\
\text{subject to}& \displaystyle\sum\limits_{v \in I} & y_v \leq 1 & \forall I \in \mathcal{I}(G),\\
\text{and}       &                                                &y_v \geq 0, &\forall v \in V;
\end{array}
\end{equation}

Let $w$ be a weight distribution, let $y_v = \frac{w(v)}{\alpha_w(G)}$ for all $v \in V$, then $y_v \geq 0$ and for all independent set $I$, we have 
$$\sum_{v \in I} y_v = \frac{w(I)}{\alpha(G_w)} \leq 1.$$
Thus, we have
$$\chi_f(G) \geq \sum_{v \in V} \frac{w(v)}{\alpha_w(G)},$$ 
then
$$\chi_f(G) \geq \sup_{w \in W} \frac{1}{\bar{\alpha}(G_w)}.$$
Conversely, let $(y_v)$ be an optimal solution of the linear program (\ref{eq:fractdualLP}). Let $w(v) = y_v$ for all $v$ in $V$, $w$ is a weight distribution (because $\chi_f(G) > 0$), and the constraint $\sum_{v \in I} y_v \leq 1$ for all $I$ in $\mathcal{I}(G)$ implies $\alpha(G_w) \leq 1$.

Thus
$$\chi_f(G) = \sum_{v \in V} y_v = w(V) \leq \frac{w(V)}{\alpha(G_w)} \leq \displaystyle\sup_{w \in W} \frac{1}{\bar{\alpha}(G_w)}.$$

\end{proof}

The unweighted case can be seen as a specific case of weighted graphs where all the vertices have weight one. Hence:

\begin{corollary}
For every graph $G$,  $\alpha^{*}(G) \leq \bar{\alpha}(G)$
\end{corollary}

The bound $\alpha^{*}(G) \leq \bar{\alpha}(G)$ is not always tight. For example, consider the path of three vertices $P_3$. Its independence ratio is $\bar{\alpha}(P_3) = \frac 2 3$ but $\alpha^{\ast}(P_3)=\frac 1 2$ (achieved by the weight distribution depicted in Figure  \ref{fig:P3weighted}).

\begin{figure}[h]
		\centering\includegraphics[width=0.3\linewidth]{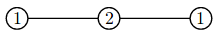}
		\caption{A weight distribution on $P_3$.}
		\label{fig:P3weighted}
\end{figure}

If $w$ is a weight distribution such that $\alpha^{*}(G) = \bar{\alpha}(G_w)$, we say that $w$ is an \emph{optimal weighting} of $G$. It follows from the definition of the fractional chromatic number as a bounded feasible linear program that there always is an optimal weighting for every graph $G$.

\subsection{Sets avoiding distance 1}
\label{sub:sets}

Let $\R^n$ be equipped with a norm $\left\| \cdot \right\|$. 
A subset $A \subset \R^n$ is said to avoid distance $1$ if:
$\forall x,y \in A, \left\| x-y \right\| \neq 1.$
The \emph{(upper) density} of a measurable subset $A \subset \R^n$ with respect to the Lebesgue measure, denoted by $\delta(A)$, is defined as:
$$\delta(A) = \limsup_{R \to +\infty} \frac{\text{Leb}(A \cap \left[ -R,R\right]^n)}{\text{Leb}(\left[ -R,R\right]^n)}.$$

We denote by $m_1(\R^n, \left\| \cdot \right\|)$ the supremum of the densities of measurable sets avoiding distance 1. As mentioned in the introduction, the following lemma is central in our work.

\begin{lemma}\cite{thesethomas}
\label{lem:alphastar}
If $G = (V,E)$ is a unit-distance graph on $\R^n$, then
$$ m_1(\R^n, \left\| \cdot \right\|) \leq \alpha^{*}(G). $$ 
\end{lemma}

\begin{proof} 
Here we adapt the proof given in \cite{AD1} for the unweighted case.

Let $R > 0$ be a real number, $w$ be a weight distribution on $V$ , and let $S$ be a set avoiding distance $1$. We consider $X_R \in \left[ -R,R\right] ^n$ chosen uniformly at random. Therefore, we have 
$$\mathbb{P}(X_R\in S) = \frac{Leb(S \cap \left[ -R,R\right] ^n)}{Leb(\left[ -R,R\right] ^n)},$$
thus, 
$$\limsup_{R \to +\infty} \mathbb{P}(X_R\in S) = \delta(S).$$
We now introduce the random variable $N = \sum_{v \in V} w(v)\mathbf{1}_{X_R+v \in S}$, its expected value is given by 
$$\mathbb{E}(N) = \sum_{v \in V} w(v) \mathbb{P}\left(X_R \in S - v \right),$$
as $V$ is finite and $\delta(S-v) = \delta(S)$ for all $v$ in $V$, we have
$$\limsup_{R \to +\infty}\mathbb{E}(N) = \sum_{v \in V}w(v)\delta(S).$$
Moreover, for all realization of $X_R$, and for all $v_1,v_2$ in $V$,
$$\left\| (X_R + v_1) - (X_R + v_2) \right\| = \left\| v_1 - v_2 \right\|,$$
and $S$ avoids distance $1$, hence for all realization of $X_R$ the set $(X_R + V) \cap S$ is an independent set of the graph $G$, therefore $N \leq \alpha(G_w)$.
Finally, we have $\delta(S) \leq \bar{\alpha}(G_w)$. This inequality is verified for every weight distribution, thus $\delta(S) \leq \alpha^{*}(G)$ for every set avoiding distance $1$ in $\R^{n}$, so we finally get

$$m_1(\R^n, \left\| \cdot \right\|) \leq \alpha^{*}(G).$$
\end{proof}

\section{Algorithm}
\label{sec:algo}

Let $v$ be a vertex of a graph $G$. The \emph{orbit} of $v$ is defined as $\left\{f(v) | f\in \text{Aut}\left( G\right) \right\},$ where $\text{Aut}\left( G\right)$ denotes the \emph{automorphism group} of $G$. Determining the orbits of a graph is a difficult problem in the general case but can be done efficiently on the geometric graphs that we study in this paper.

There is an optimal weighting that assigns the same value to every vertex of the same orbit. We call such weightings \emph{symmetric weightings}. Indeed, let $w$ be an optimal weighting distribution, and consider all its images under the automorphism group of $G$. Due to the convexity of the set of optimal weighting distributions, the average of these weightings is an optimal weighting distribution and it is symmetric.    
Let $G=(V,E)$ be a graph, let $O_1,\dots,O_p$ be its orbits and let $\mathrm{orbit}$ be a function that returns the number of the orbit of a given vertex $v$.

\subsection{General algorithm}

Our algorithm uses two linear programs that we present here. 

Given a collection of independent sets $\mathscr S=\{S_1,\dots,S_k\}$, our first program, that we call $P_1$, returns a symmetric weighting that minimizes the maximum weight of a set of $\mathscr S$. We set $n_{i,j}=|S_i\cap O_j|$. Our variables are $w_1,\dots,w_p$ where $w_j$ indicates the weight of the vertices of $O_j$.

\begin{equation}\tag{$P_1(\mathscr S)$}
\label{eq:LPdensity}
\begin{array}{ll@{}ll}
\text{minimize}  & M, &\\
\text{subject to}& \displaystyle\sum\limits_{j = 1}^{p} n_{i,j}w_j \leq M & \forall i \in \left\{1,...,k\right\},\\
                 & \displaystyle\sum_{j=1}^{p}w_j\left|O_j\right| = 1,\\
\text{and}       & w_1,...,w_p,M \in \R.
\end{array}
\end{equation}


Note that if $\mathscr S$ is the set of all independent sets in $G$, Linear Program \ref{eq:LPdensity} will return $\alpha^{\ast}(G)$ and an optimal weighting of $G$. However, this condition is not necessary and it is actually suficient that $\mathscr S$ contains an independent set of maximum weight for every symmetric weighting of $G$.

The second program, $P_2$, is an integer linear program that returns an independent set of maximum weight for a given symmetric weighting $w=(w_1,\dots,w_p)$ defined by the weight it gives to each orbit. For each vertex $v$, we use a binary variable $x_v$ that indicates whether $v$ belongs to the maximum-weight independent set $S$ we create.

\begin{equation}\tag{$P_2(w)$}
\label{eq:LPmaximalConstraint}
\begin{array}{ll@{}ll}
\text{maximize}  & \displaystyle\sum_{v\in V} w_{\textrm{orbit}(v)} x_v, &\\
\text{subject to}& x_u + x_v \leq 1 & \forall uv \in E,\\
\text{and}       & x_v \in \left\{0,1\right\} & \forall v \in V.
\end{array}
\end{equation}


Our algorithm to compute the optimal weighted independence ratio of a graph is described in Algorithm \ref{algoowir}.

\begin{algorithm}[!h]
Let $O_1,\dots,O_p$ be the orbits of vertices of $G$

Let $\mathscr S=\{\}$ and for all $j\in [1,p], w_j=\frac 1 {|V|}$

Let $\optup=1$ and $\optlow=0$.

\While {$\optup\neq\optlow$}{

Let $S$ be the independent set returned by \ref{eq:LPmaximalConstraint}.

$\mathscr S=\mathscr S\cup \{S\}$.

$\optup=w(S)$

Let $w$ be the weighting returned by \ref{eq:LPdensity} and let $\optlow$ be the objective value.

}

\textbf{return} the $w_j$ and $\optup$.
\caption{Computing an optimal weighting of a graph $G$.}
\label{algoowir}
\end{algorithm}

At each step of the algorithm, $\optup$ is an upper bound on $\alpha^{\ast}(G)$ and $\optlow$ is a lower bound. Indeed, after line 7, \ref{eq:LPmaximalConstraint} ensures that for any symmetric weight distribution $w$, there exists an independent set in $\mathscr S$ whose weight is at least $\optup$. Furthermore, after line 8, \ref{eq:LPdensity} ensures that there exists a weighting (namely, $w$) that reaches the independence ratio of $\optlow$.

At each iteration of the \textbf{while} loop except the last one, our algorithm returns an independent set $S$ for which there exists a symmetric weight distribution such that $S$ is strictly heavier than every set of $\mathscr S$. The last iteration happens when our set $\mathscr S$ contains a maximum-weight independent set for every symmetric weighting. 

At the end of the execution of the algorithm, the $w_j$ describes an optimal weighting and $\optlow=\optup=\alpha^{\ast}(G)$. 

\subsection{Some optimizations}

In Algorithm \ref{algoowir}, the most costing step is the computation of a stable set of maximal weight (Linear program \ref{eq:LPmaximalConstraint}). The computation of the weights is fast because it is a pure linear program. 

To solve the integer linear program \ref{eq:LPmaximalConstraint} faster, we add some additional constraints as follows. If $H$ is an induced subgraph, the inequality $\sum_{x \in V(H)} x_v \leq \alpha(H)$ is obviously valid for \ref{eq:LPmaximalConstraint}.

Unit-distance graphs have a small number of maximal cliques. Therefore, we replace the constraints 

\begin{eqnarray*}
x_u + x_v \leq 1, & \left\{u,v\right\} \in E,
\end{eqnarray*}
by
\begin{eqnarray*}
\displaystyle\sum_{v \in Q} x_v \leq 1, & Q \in \mathcal{Q},
\end{eqnarray*}

where $\mathcal{Q}$ denotes the set of maximal cliques, as every edge constraint is dominated by some maximal clique constraint.

Furthermore, the graphs that we build have a lot of induced Moser spindles (Figure \ref{fig:mosercolor}). 

%
\begin{figure}[h]
 	\centering\includegraphics[width=0.2\linewidth]{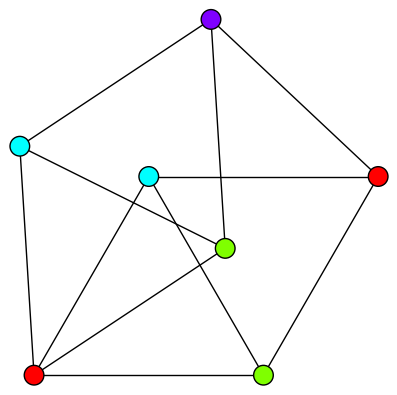}
 	\caption{For the Moser Spindle, $\alpha = 2$, $\chi = 4$ and $\alpha^* = 2/7$.}
 	\label{fig:mosercolor}
 \end{figure}   

Hence, we have added the following set of constraints:
\begin{eqnarray*}
	\displaystyle\sum_{v \in M} x_v \leq 2, & M \in \mathcal{M},
\end{eqnarray*}
where $\mathcal{M}$ denotes the set of induced Moser spindles.

These extra constraints turned out to speed up significantly the computations. For graphs which required a huge computation time, we observed a reduction of computation time up to a factor 4, as well as an important decrease of the required space.

All computations were run on computers with 2 processors of 6 cores each (Intel$^{\tiny \copyright}$ Xeon x5675 @ 3,06 GHz) with 48 Go RAM. We used IBM$^{\tiny \copyright}$ \texttt{CPLEX} (version 12) with default parameters, as integer linear programs solver. To compute the automorphism group  of a graph and enumerate the required induced subgraphs, we used \texttt{SageMath} implementations. 


\section{Construction of a graph proving Theorem \ref{th:main}}\label{sec:bound}
\label{sec:proof}

In this Section, we present our construction of a unit-distance graph improving the bound on $m_1(\R^2)$. 

The process to build an interesting unit-distance graph can be summarized as follows. We start from a chosen unit-distance graph called the \emph{starting graph}. Then we apply graph operations on it. These operations aim at lowering the optimal weighted independence ratio. Thus, two questions arise: the choice of the starting graph and the strategy of manipulation of this graph.

In Subsection \ref{sub:operations}, we introduce useful graph operations. In Subsection \ref{sub:candidate}, we explain our choice of starting graph. In Subsection \ref{sub:manipulation} we describe precisely the manipulations we did on unit-distance graphs in order to lower the bound on $m_1(\R^2)$. Finally, in Subsection \ref{sub:finalgraph}, we exhibit the unit-distance graph achieving the best bound.

\subsection{Graph operations}
\label{sub:operations}

This subsection is dedicated to the graph operations we use throughout this section. Here, we only consider unit-distance graph and the set of edges is therefore determined by the position of the vertices.

The \emph{Minkowski sum} of two unit-distance graphs is defined as follows. Let $G$ and $G'$ be two unit-distance graphs, let $n'$ be the number of vertices of $G'$ and let $u$ be a vertex of $G'$ that we call its \emph{center}. The Minkowski sum of $G$ and $G'$, denoted $G \oplus G'$, is the unit-distance graph whose vertex set is the union of $n'$ copies of $V(G)$ obtained by translating $V(G)$ by the vector $uv$ for $v\in V(G')$.

Note that because of the geometric properties of unit-distance graphs, the Minkowski sum of two graphs may contain edges that their Cartesian product does not. For example, Figure \ref{fig:rhombussum1} depicts two geometric paths of two vertices whose Cartesian product is a cycle of four vertices but because of the geometry of these graphs, we can see in Figure \ref{fig:rhombussum2} that their Minkowski sum contains five edges.



\begin{figure}[h]
	\begin{minipage}{0.5\textwidth}
		\centering\includegraphics[width=1\linewidth]{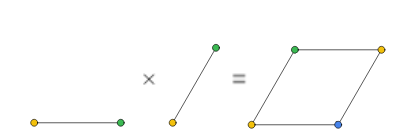}
		\caption{The classical Cartesian product.}
		\label{fig:rhombussum1}
	\end{minipage}
	\hfill
	\begin{minipage}{0.5\textwidth}
		\centering\includegraphics[width=1\linewidth]{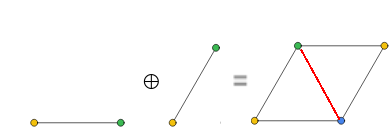}
		\caption{The Minkowski sum of the same two graphs.}
		\label{fig:rhombussum2}
	\end{minipage}
\end{figure}



Let $G=(V,E)$ be a unit-distance graph and let $u,v \in V$. The \emph{spindling} of $G$ between $u$ and $v$ contains the points of $G$ and their image by a rotation by an angle  $\theta$ around the vertex $u$ where $\theta$ is such that the distance between $v$ and its image is exactly one.

Consider the graph $G$ depicted in Figure \ref{fig:rhombus}. The graph $G$ is 3-colorable but note that in every 3-coloring of $G$, the vertices $u$ and $v$ must have the same color. Hence, the Moser spindle (Figure \ref{fig:mosercolor}) obtained by spindling $G$ between $u$ and $v$ is not 3-colorable.

\begin{figure}[h]
	\begin{minipage}{0.5\textwidth}
		\centering\includegraphics[width=0.28\linewidth]{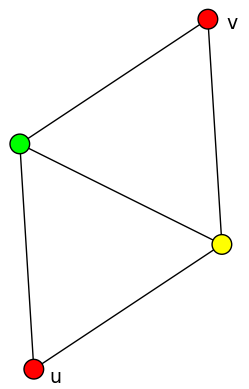}
		\caption{The graph $G$ and the two vertices $u$ and $v$.}
		\label{fig:rhombus}
	\end{minipage}
	\hfill
	\begin{minipage}{0.5\textwidth}
		\centering\includegraphics[width=0.45\linewidth]{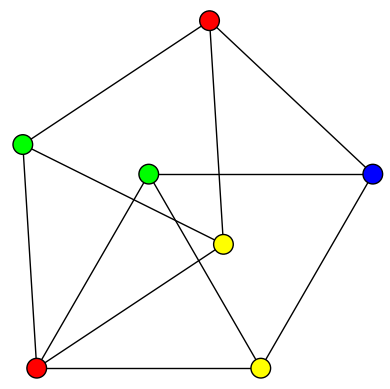}
		\caption{The Moser spindle obtained by spindling $G$ between $u$ and $v$.}
		\label{fig:spindle}
	\end{minipage}
\end{figure}

Let $G=(V,E)$ be a unit distance graph and let $r>0$. The \emph{trimming} of $G$ for the length $r$ is the unit-distance graph obtained after removing all vertices of $V$ whose Euclidean norm is greater than $r$.



\subsection{Choosing a good starting graph}
\label{sub:candidate}

We build the starting graph as follows. From a chosen set of unit-vectors viewed as complex numbers of unit modulus, we create the unit-distance graph $G$ whose vertices are the origin and the unit vectors. Then we take the Minkowski sum of the graph with itself ($G \oplus G$).

In practice, the unit vectors are chosen in the ring $R = \Z\left[e^{i\frac{\pi}{3}},e^{i\arccos\left(\frac{5}{6}\right)},e^{-i\arccos\left(\frac{5}{6}\right)} \right]$. One of the reasons we choose this ring is because it contains many interlocked Moser spindles. Such graphs are good candidates for a low optimal weighted independence ratio. 


We made several tests on different sets of unit vectors. We mainly based our choice on the numerical results we obtained after a few iterations of the process of manipulation (Subsection \ref{sub:manipulation}). An important criterion is the order of the graph we start from. Taking these conditions into account, we decided to start from de Grey's intermediate graph $W$ (see \cite{DeGrey}): a 301-vertex unit distance graph built from the Minkowski sum of a set of 30 unit vectors plus the origin. This graph actually results from a trimming of the original Minkowski sum (451 vertices), which after a few manipulations, has a similar optimal weighted independence ratio, with half as many vertices.

\subsection{Reduction and spindling of unit-distance graphs}
\label{sub:manipulation}

The strategy we choose to manipulate graphs consists in two steps. The first one is a deletion of certain vertices of the graph, giving a \emph{reduced graph} $G$. The second one is a choice of a pair of vertices $(u,v$) of the reduced graph called a \emph{spindling pair}. The graph obtained is then the spindling of $G$ between $u$ and $v$.

\paragraph{Reduction.}

Given a unit-distance graph $G$, we use Algorithm \ref{algoowir} to compute $\alpha^{*}(G)$ and an optimal weight distribution. It turns out that some orbits of an optimal solution may have zero weight. Thus we can delete these orbits and obtain a smaller graph with the same value of $\alpha^{*}$. Some of the weights given by the algorithm are "small" in comparison to the other and can be deleted with small impact on the optimal weighted independence ratio, as expressed by Lemma \ref{lem:deletion}.

\begin{lemma}
\label{lem:deletion}
Let $G=(V,E)$ be a graph, let $\varepsilon > 0$ and $w : V \rightarrow \R_+$ be a weight distribution such that $\alpha^{*}(G) = \bar{\alpha}(G_w)$. Then if $V' \subset V$ is such that $\frac{w(V')}{w(V)} = \varepsilon$, if $G'$ is the induced subgraph of $G$ whose set of vertices is $V \setminus V'$,
$$(1-\varepsilon)\alpha^{*}(G') \leq \alpha^{*}(G).$$
\end{lemma} 

\begin{proof}
Every independent set of $G'$ is an independent set of $G$. Thus, $\alpha(G'_w) \leq \alpha(G_w)$. Then, 
$$\alpha^{*}(G') \leq \frac{\alpha(G'_w)}{w(V \setminus V')} \leq \frac{\alpha(G_w)}{w(V \setminus V')} = \frac{1}{1-\varepsilon}\alpha^{*}(G).$$
\end{proof}

\paragraph{Spindling.}

The step of spindling is as follows:
\begin{itemize}
\item we choose a vertex $v$ of an orbit with greatest weight;
\item for every other vertex $v'$ of the graph, we use an integer linear program to find a maximal weighted independent set which contains $v'$ and does not contain $v$;
\item we choose a vertex whose corresponding independent set is of the smallest weight possible, we denote this vertex $u$;
\item we spindle the graph between $u$ and $v$.
\end{itemize}

As the running time of Algorithm \ref{algoowir} is lower for graphs with lots of symmetries, we get high symmetrical graphs by rotating by every angles $\frac{k\pi}{3}$ for $k \in \left[0,5\right]$. We call this operation the \emph{circling} operation. Deletion of orbits of lowest weight and circling reduces significantly the number of orbits. 

\subsection{A unit-distance graph $G$ that achieves $\chi_f(G) \leq 1999983/512933
$}
\label{sub:finalgraph}

As we have already mentioned, our starting point is de Grey's intermediate graph $W$ (see \cite{DeGrey}). We iterated the process described in Subsection \ref{sub:manipulation} seven times. The last step of our construction is a circled graph of 607 vertices. This graph is available at \url{https://www.labri.fr/perso/pecher/pmwiki/pmwiki.php/Research/AvoidingDistance1} together with a weight distribution close to the optimal that achieves an independance ratio of 512933/1999983.

\section*{Concluding remark}

In this work, we establish a new upper bound on maximal density of the sets avoiding Euclidean distance 1 in the plane and a new lower bound on its fractional chromatic number, by using a graph theoretical approach. Our algorithm is not limited to the Euclidean plane and can provide upper bounds on the maximal density of sets avoiding some distances in various normed spaces. Ongoing work studies the case of parallelohedron norms.

\acknowledgements
\label{sec:ack}

We would like to thank Christine Bachoc and Philippe Moustrou for many fruitful discussions.

\section*{Update}

By the time our paper was published in DMTCS, the bounds we establish have been improved by other authors. In \cite{AmbrusMatolcsi}, Ambrus and Matolcsi have established that $m_1(\R^2) \leq 0.25442$ with a method based on harmonic analysis very different from ours. To the best of our knowledge, one cannot infer a lower bound on $\chi_f(\R^2)$ from this result, as optimal fractional colourings could use non-measurable colour classes. 

In a polymath forum dedicated to this topic\footnote{\href{https://dustingmixon.wordpress.com/2019/03/23/polymath16-twelfth-thread-year-in-review-and-future-plans/\#comment-23601}{https://dustingmixon.wordpress.com/2019/03/23/polymath16-twelfth-thread-year-in-review-and-future-plans/\#comment-23601}}, Parts announced the construction of a 919-vertex graph that reaches $\chi_f(\R^2) \geq 3.98$. The result has not been published, but would imply that $m_1(\R^2) \leq 0.2513$ and get us increasingly close to proving Erd\H os' conjecture.

\bibliographystyle{abbrvnat}
\bibliography{article}
\label{sec:biblio}

\end{document}